\documentclass[12pt]{amsart}
\usepackage[a4paper,hmargin=3cm,vmargin=3cm]{geometry}
\usepackage{amsfonts,amssymb,amscd,amstext,amsthm}
\usepackage[pdftex]{graphicx}
\usepackage[dvips]{epsfig}
\usepackage[T1]{fontenc}
\usepackage[english]{babel}
\usepackage[utf8]{inputenc}
\usepackage[bookmarksnumbered,plainpages]{hyperref}
\usepackage{indentfirst}
\usepackage{color}
\usepackage{url}
\usepackage{comment}
\usepackage[shortlabels]{enumitem}
\usepackage{imakeidx}
\usepackage{mathrsfs}
\usepackage{stmaryrd}
\usepackage{textfit}
\usepackage{lipsum}
\usepackage{amsmath}
\usepackage{mathtools}
\usepackage{times}

\newcommand{\tr}{\mbox{tr}}
\newcommand{\spa}{\mbox{span}}
\newcommand{\Hess}{\mbox{Hess}}

\newcommand{\grad}{\mbox{grad}}

\newcommand{\co}{\colon}
\newcommand{\R}{\mathbb{R}}
\newcommand{\Sp}{\mathbb{S}}

\newcommand{\bl}{\bar\lambda}
\newcommand{\pl}{\langle}
\newcommand{\pr}{\rangle}
\newcommand{\X}{\mathfrak{X}}
\newcommand{\bla}{\bar\lambda}
\newcommand{\Hv}{\mathcal{H}}

\newcommand{\la}{\lambda}

\newtheorem{theorem}{Theorem}[section]
\newtheorem{lemma}[theorem]{Lemma}
\theoremstyle{definition}

\newtheorem{proposition}[theorem]{Proposition}

\theoremstyle{remark}
\newtheorem{remark}[theorem]{Remark}
\numberwithin{equation}{section}

\newcommand{\Addresses}{{
\bigskip 
\footnotesize
Institute of Mathematics and Computer Science \newline
University of S\~ao Paulo, S\~ao Carlos, Brazil \\
\textit{E-mail address:} \texttt{mateusrodrigues@alumni.usp.br,
manfio@icmc.usp.br}
}}

\title{Euclidean hypersurfaces with semi-parallel Moebius second fundamental form}

\date{}

\author{M. S. R. Antas and F. Manfio}

\begin{document}

\begin{abstract}
In this paper, we classify umbilic-free hypersurfaces $f\co M^{n}\to\mathbb{R}^{n+1}$, $n \geq 4,$ with semi-parallel Moebius second fundamental form and three distinct principal curvatures. 
\end{abstract}

\maketitle

{\small
\noindent {\textbf{Mathematics Subject Classification:}}\hspace*{0.1cm}
Primary 53B25, 53C30, Secondary 53C40.
}
\vspace{.2cm} \\
{\small
\noindent {\textbf{Keywords:}}\hspace*{0.1cm}
Moebius metric, Moebius shape operator, semi-parallel hypersurfaces.
}

\section{Introduction}

The study of isometric immersions $f\co M^n\to\R^{n+k}$ within the framework of Moebius geometry started with Wang \cite{Wang}, who introduced a Moebius invariant metric $\pl\,\,,\,\pr^\ast$, called the {\em Moebius metric}, and a Moebius invariant $2$-form $\beta$ on $M^n$, the {\em Moebius second fundamental form} of $f$, and proved that, when $k=1$ e $f$ is free of
umbilic points, $\pl \,\,,\,\pr^\ast$ and $\beta$ complete determines the hypersurface up to Moebius transformations. The corresponding conformal Gauss, Codazzi and Ricci equations involve two other
important Moebius invariant tensors, named the {\em Blaschke tensor} and the {\em Moebius form}.

A special class of these isometric immersions consists of the so-called {\em Moebius parallel submanifolds}. This means that the Moebius second fundamental form is parallel with respect to the normal connection, that is, $\beta$ satisfies the relation $\nabla^\perp\beta=0$. Hypersurfaces and, more generally, submanifolds in the unit sphere, whose Moebius second fundamental form is parallel, have been extensively studied and classified in \cite{classification-mobius-parallel-hyper, Zhai2, Zhai}. A natural generalization of this class consists of the so-called {\em Moebius semi-parallel submanifolds}, in the sense that the Moebius second fundamental form is semi-parallel. This means that $\beta$ satisfies the equation $R\cdot\beta=0$ (see Section \ref{sec:basics}).

In a recent paper \cite{semi}, Hu--Xie--Zhai initiated a study on Moebius semi-parallel submanifolds $f\co M^{n}\to\Sp^{n+p}$. In addition to general results involving the Blaschke tensor, the authors prove that, in the case of hypersurfaces, the number of distinct Moebius principal curvatures is at most 3. Furthermore, they obtain a classification of Moebius semi-parallel hypersurfaces with two distinct Moebius principal curvatures.

Bearing in mind the core results of \cite{semi}, a natural problem is to complete the classification of Moebius semi-parallel hypersurfaces, that is, to classify such hypersurfaces with three distinct Moebius principal curvatures. We point out that a Moebius semi-parallel hypersurface $f\co M^{n}\to \mathbb{R}^{n+1}$, $n \geq 4$, with three distinct constant Moebius principal curvatures is also Moebius parallel. 

\vspace{.2cm}

We prove the following theorem.

\begin{theorem} \label{main}
Let $f\co M^n\to\R^{n+1}$, $n\geq 3$, be a hypersurface with semi-parallel Moebius second fundamental form, with three distinct principal curvatures. Then $M^n$ is locally Moebius equivalent to a hypersurface with parallel Moebius second fundamental form, which is a cone over the torus $\mathbb{S}^{p}(r)\times \mathbb{S}^{q}(\sqrt{1-r^{2}}) \subset \mathbb{S}^{p+q+1}$, with $p+q <n$. 
\end{theorem}

The proof of Theorem \ref{main} for $n=3$ is due to Hu--Xie--Zhai \cite[Theorem 5.3]{semi}. For $n \geq 4$, our approach consists in first showing that a Moebius semi-parallel hypersurface with three distinct principal curvatures locally has constant Moebius scalar curvature, and then applying Theorem 5.5 of \cite{semi} to establish that it must be Moebius parallel.

\vspace{.2cm}

Added in proof. After an update of this paper on the arXiv, we became aware of the work of Li--Xie \cite{Li}, which independently solved the same problem using different techniques. Our approach, however, offers a streamlined framework that naturally extends to higher codimensions (see \cite[Theorem 1.4]{Antas1} to submanifolds with flat normal bundle).

\section{Preliminaries} \label{sec:basics}

%In this section we just review the basic Moebius invariants and recall the structure equations for a hypersurface $M^n$ in the Euclidean space $\R^{n+1}$.

Let $f\co M^{n}\to\R^{n+k}$ be an isometric immersion free of umbilical points. Then, the function $\rho\in C^\infty(M)$ defined by 
\[
\rho^2 = \frac{n}{n-1}(\|\alpha\|^{2}-n\|\Hv\|^2)
\]
does not vanish on $M^n$, where $\alpha$ and $\Hv$ stand for the second fundamental form and the mean curvature vector field of $f$, respectively. The metric
\begin{equation} \label{eq:metricstar}
\langle \,\,,\,\rangle^\ast=\rho^{2}\langle \,\,, \,\rangle
\end{equation}
is called the \emph{Moebius metric} determined by $f$. It was proved in \cite{Wang} that the metric \eqref{eq:metricstar} is invariant under conformal transformations of $\R^{n+k}$. 

\vspace{.2cm}

The {\em Moebius second fundamental form} of $f\co M^{n}\to\R^{n+k}$ is the symmetric bilinear traceless map $\beta$ defined by
\[
\beta(X,Y) = \rho(\alpha(X,Y)-\pl X,Y\pr\Hv),
\]
for all $X,Y\in\X(M)$. The {\em Blaschke tensor} $\psi$ of $f$ is the symmetric $C^\infty(M)$-bilinear form given by
\[
\psi(X,Y) = \frac{1}{\rho}\pl\beta(X,Y),\Hv\pr + 
\frac{1}{2\rho^2}\big(\|\grad^\ast\rho\|^2_\ast+\|\Hv\|^2\big)
\pl X,Y\pr^\ast - \frac{1}{\rho}\Hess^\ast\rho(X,Y)
\]
and its {\em Moebius form} $\omega$ is the normal bundle valued one-form defined by
\[
\omega(X) = -\frac{1}{\rho}\big(\nabla^\perp_X\Hv+
\beta(X,\grad^\ast\rho)\big),
\]
where $\grad^\ast$ and $\Hess^\ast$ denote the gradient and the Hessian relative to the metric $\pl \,\,,\,\pr^\ast$.

\vspace{.2cm}

In the special case of a hypersurface $f\co M^n\to\R^{n+1}$, choose a local smooth unit normal vector field $N$ along $f$ and denote by $A$ the shape operator of $f$ with respect to $N$, which is given by
\[
\pl AX,Y\pr = \pl\alpha(X,Y),N\pr,
\]
for all $X,Y\in\X(M)$. Consider now the {\em Moebius shape operator} $B$ associated to $\beta$, given by
\[
\pl BX,Y\pr^\ast = \pl\beta(X,Y),N\pr,
\]
for all $X,Y\in\X(M)$. In this case, $B$ is explicitly given by 
\begin{equation} \label{eq:BrelA}
B=\rho^{-1}(A-HI),
\end{equation}
where $H=\pl\Hv,N\pr$ is the mean curvature function of $f$. The Blaschke tensor $\psi$ and the Moebius form $\omega$ can be written as
\[
\hat{\psi}X=\frac{H}{\rho}BX+\frac{1}{2\rho^{2}}(\|\grad^{*}\rho\|^{2}_{*}+H^{2})X-\frac{1}{\rho}\nabla^{*}_{X}\grad^{*}\rho,
\]
and
\[
\omega(X)=-\frac{1}{\rho}\langle \grad^{*}H+B\grad^{*}\rho,X\rangle^{*},
\] 
for all $X\in\mathfrak{X}(M)$. Moreover, the conformal Gauss, Codazzi and Ricci equations for $f\co M^n\to\R^{n+1}$ are, respectively
\begin{equation}\label{eq:confGauss}
R^{*}(X,Y)=BX\wedge^{*}BY+\hat{\psi}X \wedge^{*} Y+X\wedge^{*}\hat{\psi}Y,
\end{equation} 
\begin{eqnarray} \label{eq:ConfCodazzi}
\begin{aligned}
(\nabla^{*}_{X}B)Y-(\nabla^{*}_{Y}B)X&=\omega(X)Y-\omega(Y)X,\\
(\nabla^{*}_{X}\hat{\psi})Y-(\nabla^{*}_{Y}\hat{\psi})X&=\omega(Y)BX-\omega(X)BY,
\end{aligned}
\end{eqnarray}
and
\begin{equation}
d\omega(X,Y)=\langle [\hat{\psi},B]X,Y\rangle^{*},
\end{equation}
where $\wedge^{*}$ and $R^{*}$ stand for the wedge product and the curvature tensor, respectively, with respect to $\pl \,\,,\,\pr^\ast$.

\vspace{.2cm}

Let $\lambda_1,\ldots,\lambda_n$ be the principal curvatures of $f$ with respect to the shape operator $A$, and denote by $\bl_1,\ldots,\bl_n$ the \emph{Moebius principal curvatures} of $f$, with respect to the Moebius shape operator $B$. It follows from \eqref{eq:BrelA} that
\[
\bla_i=\rho^{-1}(\lambda_i-H),
\]
$1\leq i\leq n$. In particular, the number and multiplicities of the principal curvatures $\bar\lambda_i$ are the same as those of the principal curvatures $\lambda_i$.

\vspace{.2cm}

Recall that an isometric immersion $f\co M^{n}\to\R^{n+k}$ is {\em semi-parallel} if $R\cdot\alpha=0$, where
\begin{eqnarray*}
\begin{aligned}
(R\cdot\alpha)(X,Y,Z,W) & =  R^\perp(X,Y)\alpha(Z,W) - 
\alpha(R(X,Y)Z,W) \\ & - \alpha(Z,R(X,Y)W),
\end{aligned}
\end{eqnarray*}
for all $X,Y,Z,W\in\X(M)$. When $f\co M^n\to\R^{n+1}$ is Moebius umbilical-free hypesurface, the above condition, expressed in terms of the Moebius shape operator $B$, becomes:
\[
(R(X,Y)\cdot B)Z = R(X,Y)BZ - B(R(X,Y)Z),
\]
for all $X,Y,Z\in\X(M)$. We point out that the Moebius form of a Moebius semi-parallel umbilical-free hypesurface $f\co M^n\to\R^{n+1}$, $n\geq3$, is always closed (see \cite[Proposition 3.2]{semi}). Therefore, it follows from the conformal Ricci equation that there exists an orthonormal tangent frame $X_1,\ldots,X_n$ on $M^n$ with respect to $\pl\,\,,\,\pr^\ast$, and functions $\theta_i\in C^\infty(M)$, $1\leq i\leq n$, such that
\begin{equation} \label{eq:theta_i}
BX_i=\bar{\lambda}_iX_i \quad \mbox{and} \quad \hat{\psi}X_i=\theta_iX_i.
\end{equation}
%In Proposition 5.1 of \cite{semi} was proved that a Moebius semi-parallel umbilical-free hypersurface has at most three distinct Moebius principal curvatures. 

The next result is a simple characterization of Moebius semi-parallel hypersurfaces.

\begin{lemma} \label{condition}
Let $f\co M^{n}\to\mathbb{R}^{n+1}$, $n\geq 3$, be a hypersurface free of umbilic points. Then $f$ is Moebius semi-parallel hypersurface if and only if
\[
\bar{\lambda}_i \bar{\lambda}_j+\theta_i+\theta_j=0, \quad 1 \le i \ne j \le 3,
\]
where $\bar{\lambda}_1, \bar{\lambda}_2$ and $\bar{\lambda}_3$ are the distinct Moebius principal curvatures of $f$.
\end{lemma}
\begin{proof}
Notice that $f$ is Moebius semi-parallel hypersurface if and only if
\[
\langle R^{*}(X_i,X_j)X_k, BX_l\rangle^{*}+\langle BX_k, R^{*}(X_i,X_j)X_l\rangle^{*}=0,
\]
that is,
\[
\langle R^{*}(X_i,X_j)X_k,X_l\rangle^{*}(\bar{\lambda}_l-\bar{\lambda}_k)=0,
\]
for all $1\leq i,j,k,l\leq n$. Taking $j=k$ and $i=l$, with $1 \le i \ne j \le n$, we get
\[
\langle R^{*}(X_i,X_j)X_j,X_i\rangle (\bar{\lambda}_i-\bar{\lambda}_j)=0.
\]
Therefore, we have $K^{*}(X_i,X_j)=0$ or $\bar{\lambda}_i=\bar{\lambda}_j$, for all $1 \le i \ne j\le n$. The conclusion follows from the conformal Gauss equation \eqref{eq:confGauss}.
\end{proof}

In particular, Lemma \ref{condition} shows that the number of distinct eigenvalues of $\hat{\psi}$ is at most three.

\section{Some basic lemmas}

In this section, we will establish some properties of the eigenbundle $E_{\la_i}$ associated with the Moebius principal curvatures, which will be used in the proof of Theorem \ref{main}. Let $f\co M^{n}\to\mathbb{R}^{n+1}$, $n\geq 4$, be an umbilic-free hypersurface with semi-parallel Moebius second fundamental form, with three distinct Moebius principal curvatures $\bar{\lambda}_1, \bar{\lambda}_2$, $\bar{\lambda}_3$ of multiplicities $m_1, m_2$ and $m_3$, respectively. Since $n\geq4$, one of these $m_i$ is necessarily greater than one, say, $m_3\geq2$. 

\vspace{.2cm}

Regarding the functions $\theta_i$ given in \eqref{eq:theta_i}, the functions
\[
\bar{\lambda}_i^{2}+2\theta_i
\]
are nowhere vanishing, for $1\leq i\leq 3$. In fact, if $\bar{\lambda}_1^{2}+2\theta_1$ is zero at a point $p \in M^n$, from Lemma \ref{condition} one has
\begin{align*}(\bar{\lambda}_2-\bar{\lambda}_1)(\bar{\lambda}_3-\bar{\lambda}_1)&=\bar{\lambda}_2\bar{\lambda}_3-\bar{\lambda}_2 \bar{\lambda}_1-\bar{\lambda}_1 \bar{\lambda}_3 + \bar{\lambda}_1^{2}\\
&=-\theta_2-\theta_3+\theta_2+\theta_1+\theta_1+\theta_3+\bar{\lambda}_1^{2}=0,
\end{align*} which implies $\bar{\lambda}_2=\bar{\lambda}_1$ or $\bar{\lambda}_3=\bar{\lambda}_1$, and this is a contradiction.

%\vspace{.2cm}

%We denote by $E_{\lambda_i}$ the eigenbundle distribution associated with $\bar\lambda_i$, for $1\leq i\leq3$.

%Since the Moebius form of $f$ is closed, there exists an orthonormal frame $\{X_1,\ldots,X_n\}$ with respect to $\langle \cdot, \cdot \rangle^{*}$ such that
%\[
%BX_i=\bar{\lambda}_j X_i \quad \mbox{and} \quad \hat{\psi}X_i=\theta_jX_i,
%\]
%for each $1\leq j\leq 3$ and for all $m_1+\ldots+m_{j-1}+1 \le i \le m_1+\ldots+m_j$.

%\begin{lemma}
%The functions $\bar{\lambda}_i^{2}+2\theta_i$ are nowhere vanishing, for $1\leq i\leq 3$.
%\end{lemma}
%\begin{proof}
%Assume otherwise that $\bar{\lambda}_1^{2}+2\theta_1$ is zero at $p \in M^n$. Then, from Lemma \ref{condition} one has
%\begin{align*}(\bar{\lambda}_2-\bar{\lambda}_1)(\bar{\lambda}_3-\bar{\lambda}_1)&=\bar{\lambda}_2\bar{\lambda}_3-\bar{\lambda}_2 \bar{\lambda}_1-\bar{\lambda}_1 \bar{\lambda}_3 + \bar{\lambda}_1^{2}\\
%&=-\theta_2-\theta_3+\theta_2+\theta_1+\theta_1+\theta_3+\bar{\lambda}_1^{2}=0,
%\end{align*} which implies $\bar{\lambda}_2=\bar{\lambda}_1$ or $\bar{\lambda}_3=\bar{\lambda}_1$, and this is a contradiction.
%\end{proof}

\begin{lemma}\label{warped}
If $\dim E_{\lambda_i} \geq 2$, then $X_k(\bar{\lambda}_i^2+2\theta_i)=0$, for all $X_k \in E_{\lambda_i}.$
\end{lemma}
\begin{proof}
Let $X_k,X_\ell \in E_{\lambda_i}$ with $k\ne\ell$. Taking the inner product of the conformal Codazzi equations \eqref{eq:ConfCodazzi} with $X_\ell$ gives
%Then, the inner product of $X_\ell$ with the conformal Codazzi equations  \begin{align*}(\nabla^{*}_{X_k}B)X_\ell-(\nabla^{*}_{X_\ell}B)X_k&=\omega(X_k)X_\ell-\omega(X_\ell)X_k,\\ (\nabla^{*}_{X_k}\hat{\psi})X_\ell-(\nabla^{*}_{X_\ell}\hat{\psi})X_k&=\omega(X_\ell)BX_k-\omega(X_k)BX_\ell\end{align*} yields 
\begin{equation*}
        X_k(\bar{\lambda}_i)=\omega(X_k) \quad \mbox{and}  \quad X_k(\theta_i)=-\bar{\lambda}_i\omega(X_k).
\end{equation*}
The result follows by multiplying the first equation by $\bar{\lambda}_i$ and adding it to the second one.
%, we obtain $X_k(\bar{\lambda}_i^{2}+2\theta_i)=0$.
\end{proof}

\begin{lemma}\label{umbilical}
The distribution $E_{\lambda_i}$ is umbilical with mean curvature vector field
\[
-(\grad^{*}(\log |\bar{\lambda}_i^{2}+2\theta_i|^{-\frac{1}{2}}))_{E_{\lambda_i}^{\perp}},
\]
for all $1 \le i \le 3$.
\end{lemma}
\begin{proof}
Let $X_u,X_v \in E_{\lambda_i}$ and $X_z \in E_{\lambda_i}^{\perp}$. Taking the inner product of the conformal Codazzi equations \eqref{eq:ConfCodazzi} with $X_v$ gives
%Then, the inner product of $X_v$ with the conformal Codazzi equations 
%    \begin{align*}
%        (\nabla^{*}_{X_u}B)X_z-(\nabla^{*}_{X_z}B)X_u&=\omega(X_u)X_z-\omega(X_z)X_u,\\
 %       (\nabla^{*}_{X_u}\hat{\psi})X_z-%(\nabla^{*}_{X_z}\hat{\psi})X_u&=\omega(X_z)BX_u-%\omega(X_u)BX_z 
%    \end{align*} yields 
\begin{align}
        (\bar{\lambda}_z-\bar{\lambda}_i)\langle \nabla^{*}_{X_u}X_z,X_v\rangle^{*}-X_z(\bar{\lambda}_i)\langle X_u,X_v\rangle^{*}&=-\langle X_u,X_v\rangle^{*}\omega(X_z),\label{cd1}\\
        (\theta_z-\theta_i)\langle \nabla^{*}_{X_u}X_z,X_v\rangle^{*}-X_z(\theta_i)\langle X_u,X_v\rangle^{*}&=\bar{\lambda}_i \langle X_u,X_v\rangle^{*}\omega(X_z)\label{cd2}.
\end{align}
Multiplying \eqref{cd1} by $\bar{\lambda}_i$, adding it to \eqref{cd2} and using Lemma \ref{condition}, we obtain
\[
(-\bar{\lambda}_i^{2}-2\theta_i)\langle \nabla^{*}_{X_u}X_z,X_v\rangle^{*}-\frac{X_z(\bar{\lambda}_i^{2}+2\theta_i)}{2}\langle X_u,X_v\rangle^{*}=0.
\]
This implies $$\langle \nabla_{X_u}X_v,X_z\rangle^{*}=-\langle X_u,X_v\rangle^{*}\langle \grad^{*}(\log |\bar{\lambda}_i^{2}+2\theta_i|^{-\frac{1}{2}}), X_z\rangle^{*},$$
and this concludes the proof.
\end{proof}

\begin{lemma}\label{integrability}
    $E_{\lambda_i}^{\perp}$ is an integrable distribution for all $1 \le i \le 3.$
\end{lemma}
\begin{proof}
    Let $X_u,X_v \in E_{\lambda_i}^{\perp}$ and $X_z \in E_{\lambda_i}.$ Taking the inner product of $X_z$ with the conformal Codazzi equation $$(\nabla^{*}_{X_u}B)X_v-(\nabla^{*}_{X_v}B)X_u=\omega(X_u)X_v-\omega(X_v)X_u,$$ we get $$(\bar{\lambda}_v-\bar{\lambda}_i)\langle \nabla^{*}_{X_u}X_v,X_z\rangle^{*}=(\bar{\lambda}_u-\bar{\lambda}_i)\langle \nabla^{*}_{X_v}X_u,X_z\rangle^{*}.$$ If $\bar{\lambda}_u=\bar{\lambda}_v,$ then $[X_u,X_v]\in E_{\lambda_i}^{\perp}.$ If $\bar{\lambda}_v \ne \bar{\lambda}_u,$ then multiplying the preceding equation by $\bar{\lambda}_v$ and using the Lemma \ref{condition}, we obtain \begin{equation}\label{vcd1}(\bar{\lambda}_v^{2}+\theta_v+\theta_i)\langle \nabla^{*}_{X_u}X_v,X_z\rangle^{*}=(\theta_i-\theta_u)\langle \nabla^{*}_{X_v}X_u,X_z\rangle^{*}.\end{equation}
However, the inner product of $X_z$ with the equation $$(\nabla^{*}_{X_u}\hat{\psi})X_v-(\nabla^{*}_{X_v}\hat{\psi})X_u=\omega(X_v)BX_u-\omega(X_u)BX_v$$ yields 
\begin{equation}\label{vcd2}
    (\theta_v-\theta_i)\langle \nabla^{*}_{X_u}X_v, X_z\rangle^{*}=(\theta_u-\theta_i)\langle \nabla^{*}_{X_v}X_u,X_z\rangle^{*}.
\end{equation}
Adding equations \eqref{vcd1} and \eqref{vcd2} we obtain $(\bar{\lambda}_v^{2}+2\theta_v)\langle \nabla^{*}_{X_u}X_v,X_z\rangle^{*}=0,$ which implies $\langle \nabla^{*}_{X_u}X_v,X_z\rangle^{*}=0$, and the conclusion follows.
\end{proof}

A central step in proving Theorem \ref{main} is to obtain a warped product structure with respect to the metric \eqref{eq:metricstar}. To this end, we recall that an orthogonal net $\mathcal{E}=\{E_1,\ldots,E_k\}$ on a Riemannian manifold $M^{n}$ is called {\em twisted produd net} if $E_i$ is umbilical and $E_i^{\perp}$ is integrable, for all $1\leq i\leq k$. Thus, it follows from Lemmas \ref{umbilical} and \ref{integrability} that
\[
\mathcal{E}=\{E_{\lambda_1}, E_{\lambda_2}, E_{\lambda_3}\}
\]
is a twisted produd net on $(M^n,\pl \,\,,\,\pr^\ast)$.

\begin{proposition} \label{prop:warp}
Let $f\co M^{n}\to\R^{n+1}$, $n\geq 4$, be an umbilic-free hypersurface with semi-parallel Moebius second fundamental form, with three distinct Moebius principal curvatures. Then there exists an open subset $U\subset M$ such that the Moebius metric \eqref{eq:metricstar} is a warped product metric on $U$. 
\end{proposition}
\begin{proof}
Fix an arbitrary point $x\in M^n$. It follows from \cite[Corollary 1]{decomposition} that there exists an open subset $U\subset M^{n}$, with $x\in U$, and a product representation $\Phi\co\prod_{i=1}^{3}M_i \to U$ of $\mathcal{E}$ that is an isometry with respect to a twisted product metric 
\begin{equation} \label{mobius-metric1}
\pl\,\,,\,\pr = \sum_{i=1}^{3}\rho_i^{2} \pi_i^{*} \pl \,\,,\,\pr_{i},
\end{equation}
for some positive \emph{twisting functions} $\rho_i \in C^{\infty}\left(\prod_{i=1}^{3}M_i\right)$, $1\leq i\leq3$, where $\pi_i$ denotes the canonical projection. Let $\{E_1,\ldots,E_3\}$ be the product net on $\prod_{i=1}^{3}M_i$, that is,
\[
E_i(x)={\tau_i^{x}}_{*}(T_{x_i}M_i),
\]
for all $x=(x_1, x_2,x_3)\in \prod_{i=1}^{3}M_i$, where $\tau_i^{x}:M_i\to \prod_{i=1}^{3}M_i$ is the standard inclusion. It follows from \cite[Proposition 2]{decomposition} that $E_i$ is an umbilical distribution with mean curvature vector field
\[
-(\grad(\log \circ \rho_i))_{E_i^{\perp}},
\]
where $\grad$ is the gradient with respect to $\pl \,\,,\,\pr$. We claim that there exist or\-tho\-go\-nal coordinates $(x_1,\ldots,x_n)$ on $\prod_{k=1}^{3}M_k$ and positive functions $r_i\in C^\infty(M_i)$ such that $$\rho_i=r_i|\bar{\lambda}_{i}^{2}+2\theta_i|^{-\frac{1}{2}}\circ \Phi,$$ for each $1 \le i \le 3.$ In fact,
%in Lemma \ref{umbilical} we proved that $E_{\lambda_i}$ is an umbilical distribution with mean curvature vector field $-(\grad^{*}(\log |\bar{\lambda}_i^{2}+2\theta_i|^{-\frac{1}{2}}))_{E_{\lambda_i}^{\perp}}$. 
since $\Phi$ is an isometry and satisfy $\Phi_{*}E_i(x)=E_{\lambda_i}(\Phi(x))$, we have  \begin{align*}
    -\langle \grad^{*}(\log |\bar{\lambda}_i^{2}+2\theta_i|^{-\frac{1}{2}}), X_\ell\rangle^{*}&=\langle \nabla^{*}_{X_a}X_a, X_\ell\rangle^{*}\\
    &=\langle\nabla^{*}_{{\tau_{i}^{x}}_{*}\bar{X}_a}\Phi_{*}{\tau_{i}^{x}}_{*}\bar{X}_a, \Phi_{*}{\tau_{\ell}^{x}}_{*}\bar{X}_\ell \rangle^{*}\\
    &=\langle\Phi_{*}\nabla_{{\tau_{i}^{x}}_{*}\bar{X}_a}{\tau_{i}^{x}}_{*}\bar{X}_a, \Phi_{*}{\tau_{\ell}^{x}}_{*}\bar{X}_\ell \rangle^{*}\\
    &=\langle\nabla_{{\tau_{i}^{x}}_{*}\bar{X}_a}{\tau_{i}^{x}}_{*}\bar{X}_a,{\tau_{\ell}^{x}}_{*}\bar{X}_\ell \rangle\\
    &=-\langle \grad(\log \circ \rho_i), {\tau_{\ell}^{x}}_{*}\bar{X}_\ell\rangle,
\end{align*} 
for all $X_a\in E_{\lambda_i}$ and $X_\ell \in E_{\lambda_i}^{\perp}$. Thus, along the orthogonal complement of the subspace $E_{\lambda_i}(\Phi(x))$, we have
\begin{align*}
\Phi_{*} \grad(\log \circ \rho_i) & = 
\grad^{*}(\log |\bar{\lambda}_i^{2}+2\theta_i|^{-\frac{1}{2}}) \\
& = \Phi_{*}\grad(\log \circ |\bar{\lambda}_i^{2}+2\theta_i|^{-\frac{1}{2}}\circ \Phi).
\end{align*}
Since $E_{\lambda_1}^{\perp}, E_{\lambda_2}^{\perp}$ and $E_{\lambda_3}^{\perp}$ are integrable, it follows that $f$ is locally holonomic on $U$, and therefore there exist orthogonal coordinates $(x_1, \ldots, x_n)$ on $U$ such that 
\[
E_{\lambda_i}=\spa\left\{\frac{\partial}{\partial x_{m_1+\ldots+m_{i-1}+1}}, \ldots, \frac{\partial}{\partial x_{m_1+\ldots+m_i}}\right\},
\]
for all $1 \le i \le 3.$ Therefore, one has
\[
\rho_i=r_i|\bar{\lambda}_i^{2}+2\theta_i|^{-\frac{1}{2}}\circ \Phi,
\]
for some positive functions $r_i\in C^\infty(M_i)$, and this proves the claim. Dropping the isometry $\Phi$, we can write the metric \eqref{mobius-metric1} as 
\begin{eqnarray} \label{mobius-metric2}
\begin{aligned}    
\langle \cdot, \cdot \rangle^{*} = & |\bar{\lambda}_1^{2}+2\theta_1|^{-1}\sum_{j=1}^{m_1}V_j^{2}dx_j^{2}+|\bar{\lambda}_2^{2}+2\theta_2|^{-1}\sum_{j=m_1+1}^{m_1+m_2}V_j^{2}dx_j^{2} \\
&+|\bar{\lambda}_3^{2}+2\theta_3|^{-1}\sum_{j=m_1+m_2+1}^{n}V_j^{2}dx_j^{2},
\end{aligned}
\end{eqnarray} 
where
\[
V_j=r_i\sqrt{\left\langle\frac{\partial}{\partial x_j}, \frac{\partial}{\partial x_j}\right\rangle},
\]
for $1\leq i\leq3$ and $m_1+\ldots+m_{i-1}+1 \le j \le m_1+\ldots+m_i$. Setting
\[
v_i=|\bar{\lambda}_j^{2}+2\theta_j|^{-\frac{1}{2}}V_i,
\]
for $1\le j\le3$ and $m_1+\ldots+m_{j-1}+1 \le i \le m_1+\ldots+m_j$, we can write
\[
\langle \cdot, \cdot \rangle^{*}=\sum_{i=1}^{n}v_i^{2}dx_i^{2},
\]
and $\{X_i: X_i=v_i^{-1}\frac{\partial}{\partial x_i}\}_{i=1}^{n}$ is an orthonormal frame on $(M^{n}, \langle \,\,,\,\rangle^{*})$. It follows from Lemma \ref{condition} and the conformal Gauss equation, that $f$ is Moebius semi-parallel if and only if
\[
K^{*}(X_a,X_b)=0,
\]
for all $X_a \in E_{\lambda_i}$ and $X_b \in E_{\lambda_j}$, where $1 \le i \ne j \le 3$. Finally, from Lemma \ref{warped}, we have $\frac{\partial \mu}{\partial x_i}=0$ for all $m_1+m_2+1 \le i \le n$, and this implies that $\langle \cdot, \cdot \rangle^{*}$ is a warped metric with warping function $\mu=|\bar{\lambda}_3^{2}+2\theta_3|^{-\frac{1}{2}}$.
\end{proof}

\section{Proof of Theorem \ref{main}}

\begin{proof}[Proof of Theorem \ref{main}]
First, assume that $m_1=1$, $m_2=1$ and $m_3 \geq 2$. By a change of coordinates, we can assume that the Moebius metric \eqref{mobius-metric2} can be written as 
\begin{equation}\label{mobius-metric-2}
    \langle \,\,,\,\rangle^{*}=|\bar{\lambda}_1^{2}+2\theta_1|^{-1}dx_1^{2}+|\bar{\lambda}_2^{2}+2\theta_2|^{-1}dx_2^2+|\bar{\lambda}_3^{2}+2\theta_3|^{-1}\sum_{j=3}^{n}V_j^{2}dx_j^{2}.
\end{equation} 
By fixing the coordinate $x_2$ on the open subset $U$, given by the Proposition \ref{prop:warp}, the product $M_1\times M_3$ endowed with the metric
\[
|\bar{\lambda}_1^{2}+2\theta_1|^{-1}dx_1^2+|\bar{\lambda}_3^{2}+2\theta_3|^{-1}\sum_{j=3}^{n}V_j^{2}dx_j^{2}
\]
is a totally geodesic hypersurface of $M=M_1\times M_2 \times M_3$.
%endowed with the metric \eqref{mobius-metric-2}. In fact, let $c:M_1\times M_3 \to M_1\times M_2 \times M_3$ be the canonical inclusion and denote by $\nabla^{*}$ and $\Tilde{\nabla}$ the Levi-Civita connections of the Moebius metric and the product metric on $M$. By Proposition 1 of \cite{decomposition}, 
%\begin{equation}\label{connection}\nabla^{*}_{X}Y=\Tilde{\nabla}_{X}Y+\sum_{i=1}^{3}(\langle X^i, Y^i\rangle^{*}U_i-\langle X,U_i\rangle^{*}Y^i-\langle Y,U_i\rangle^{*}X^{i}),\end{equation} where $X\to X^i$ is the orthogonal projection over $E_{\lambda_i}$ and $U_i=-\grad^{*}(\log \circ |\bar{\lambda}_i^{2}+2\theta_i|^{-\frac{1}{2}})$, $1 \le i \le 3.$ Then $\alpha^{c}(X_a,X_b)=(\nabla^{*}_{X_a}c_{*}X_b)_{E_{\lambda_2}}=0$ for all $X_a, X_b \in E_{\lambda_1}\cup E_{\lambda_3}$, because $|\bar{\lambda}_i^{2}+2\theta_i|^{-\frac{1}{2}}\circ c$ does not depend on $x_2$.  Applying the Gauss equation to $c$, we conclude that $K^{c}(X_a,X_b)=K^{*}(X_a,X_b)$.
By Lemma \ref{warped}, the induced metric by the canonical inclusion $i_2\co M_1\times M_3\to M_1\times M_2\times M_3$ is warped with warping function
\[
\mu=|\bar{\lambda}_3^{2}+2\theta_3|^{-\frac{1}{2}}.
\]
Since $f$ is Moebius semi-parallel, one has $\Hess \mu=0$, where the hessian is computed with respect to the metric $g_1=|\bar{\lambda}_{1}^{2}+2\theta_1|^{-1}dx_1^{2}$. Denote by
\[
v_1=|\bar{\lambda}_1^{2}+2\theta_1|^{-\frac{1}{2}} \ \text{and} \ v_i=|\bar{\lambda}_3^{2}+2\theta_3|^{-\frac{1}{2}}V_i,
\]
for $3\le i\le n$, and set
\[
h_{ij}=\frac{1}{v_i}\frac{\partial v_j}{\partial x_i}
\]
for $2\le i\ne j\le n$. In order to compute $\Hess\,\mu (\partial_1,\partial_1)$, we have
\[
\nabla_{\partial_1}\partial_1=\nabla_{\partial_1}(v_1X_1) =\frac{\partial v_1}{\partial x_1}X_1
=\frac{\partial v_1}{\partial x_1}\frac{1}{v_1}\frac{\partial}{\partial x_1},
\]
with respect to $g_1$. Thus,
\begin{align*}
0 &= \Hess \,\mu (\partial_1, \partial_1)=\frac{\partial}{\partial x_1}\left(\frac{\partial \mu}{\partial x_1}\right)-(\nabla_{\partial_1}\partial_1)(\mu) \\
&=\frac{1}{V_j}\frac{\partial}{\partial x_1}(v_1h_{1j})-\frac{1}{v_1}\frac{\partial v_1}{\partial x_1}\frac{1}{V_j}\frac{\partial v_j}{\partial x_1} =\frac{\partial}{\partial x_1}\left(\frac{v_1}{V_j}h_{1j}\right)-\frac{1}{V_j}\frac{\partial v_1}{\partial x_1}h_{1j}\\
    &=\frac{v_1}{V_j}\frac{\partial}{\partial x_1}h_{1j}.
\end{align*}
Since
\[
h_{1j}=-\frac{V_j}{2}\frac{|\bar{\lambda}_1^{2}+2\theta_1|^{\frac{1}{2}}}{|\bar{\lambda}_3^{2}+2\theta_3|^{\frac{3}{2}}}\frac{\partial}{\partial x_1}|\bar{\lambda}_3^{2}+2\theta_3|
\]
and $\frac{\partial}{\partial x_1} h_{1j}=0$, we conclude that
\begin{equation} \label{eq:delx_1}
\frac{\partial}{\partial x_1}|\bar{\lambda}_3^{2}+2\theta_3|=0
\end{equation}
along $M_1\times M_3$. Since $x_2\in M_2$ is arbitrary, equation \eqref{eq:delx_1} holds along $M_1\times M_2 \times M_3$. A similar argument shows that
\[
\frac{\partial}{\partial x_2}|\bar{\lambda}_3^{2}+2\theta_3|=0
\]
along $M_1\times M_2 \times M_3$, and Lemma \ref{warped} implies that $|\bar{\lambda}_3^{2}+2\theta_3|$ is constant in $M_1\times M_2 \times M_3.$ It follows from the conformal Gauss equation and Lemma \ref{condition} that the Moebius scalar curvature is constant. According to \cite[Theorem 5.5]{semi}, the Moebius second fundamental form of $f$ must be parallel. Finally, from \cite{classification-mobius-parallel-hyper}, we conclude that $f(M^{n})$ is locally Moebius equivalent to a cone over the torus $\mathbb{S}^{1}(r)\times \mathbb{S}^{1}(\sqrt{1-r^{2}})\subset \mathbb{S}^{3}$.
\end{proof}

\begin{remark}
We note that the proof of Theorem \ref{main} in the cases $m_1 \ge 1$, $m_2 \ge 2$, and $m_3 \ge 2$ reduces to showing that $\vert{}\bar{\lambda}_i^{2}+2\theta_i\vert{}$ is constant on $M_1 \times M_2 \times M_3$ for $m_i \ge 2$. Since the arguments required are entirely similar to the ones presented above, this concludes the proof of Theorem \ref{main}.    
\end{remark}

\section*{Acknowledgements}

Mateus Antas was supported by CAPES, grant 88887.133756/2025-00. Fernando Manfio was supported by FAPESP, grant 2022/16097-2.

\Addresses

\end{document}